\documentclass[a4paper,12pt]{amsart}
\usepackage{amssymb}
\usepackage{txfonts}

\usepackage{euscript,eufrak,verbatim}
\usepackage{graphicx}
\usepackage{epstopdf}
\usepackage{color}

\newtheorem{theorem}{Theorem}[section]
\newtheorem{lemma}[theorem]{Lemma}

\newtheorem{remark}{Remark}[section]

\def\N{\mathbb{N}}

\begin{document}

\title{The multifractal spectra of  V-statistics}

\date{}

\author[Aihua Fan]{Aihua Fan}

\author[J\"{o}rg Schmeling]{J\"{o}rg Schmeling}

\author[Meng Wu]{Meng Wu}

\address[Aihua Fan]{LAMFA, UMR 7352 CNRS, University of Picardie,
33 rue Saint Leu, 80039 Amiens, France}\email{ai-hua.fan@u-picardie.fr}

\address[J\"{o}rg Schmeling]{MCMS,
Lund Institute of Technology, Lund University
Box 118
SE-221 00 Lund, Sweden}
\email{joerg@maths.lth.se}

\address[Meng Wu]{LAMFA, UMR 7352 CNRS, University of Picardie,
33 rue Saint Leu, 80039 Amiens, France}
\email{meng.wu@u-picardie.fr}

\begin{abstract}
 Let $(X, T)$ be a topological dynamical system and let $\Phi: X^r \to \mathbb{R}$
 be a continuous function on the product space $X^r= X\times \cdots \times X$ ($r\ge 1$).
 We are interested in the limit of V-statistics taking $\Phi$ as kernel:
\[
\lim_{n\to \infty} n^{-r}\sum_{1\le i_1, \cdots, i_r\le n} \Phi(T^{i_1}x, \cdots, T^{i_r} x).
\]
The multifractal spectrum of topological entropy of the above limit is expressed by a variational
principle when the system satisfies the specification property.  Unlike the classical case ($r=1$)
where the spectrum is an analytic function when $\Phi$ is H\"{o}lder continuous, the spectrum of the limit of  higher order
 V-statistics ($r\ge 2$) may be  discontinuous even for very nice kernel $\Phi$.
\end{abstract}

\maketitle

\section{Introduction}

Consider a topological dynamical system $(X, T)$, where $T: X \to X$ is a continuous transformation
on a compact metric space $X$ with metric $d$.  For $r\ge 1$, let $X^r = X\times \cdots \times X$ (product of $r$
copies of $X$) and let $C(X^r)$ be the space of
continuous functions $\Phi: X^r \to \mathbb{R}$.

 For $\Phi \in C(X^r)$ and $n\ge 1$, let
 $$
          V_\Phi(n, x) = n^{-r}\sum_{1\le i_1, \cdots, i_r\le n} \Phi(T^{i_1}x, \cdots, T^{i_r} x)
 $$
 and
 $
         V_\Phi(x) = \lim_{n\to \infty } V_\Phi(n, x)
 $
 if the limit exists. For $\alpha \in \mathbb{R} $, define
 $$
        E_\Phi(\alpha) =\left\{x \in X:  \lim_{n\to \infty } V_\Phi(n, x) = \alpha\right\}.
 $$

 The problem treated in the present paper is to measure the sizes of the sets $E_\Phi(\alpha)$.
 To measure the sizes of the sets $E_\Phi(\alpha)$,
 we adopt the notion of topological entropy introduced by Bowen
(\cite{Bowen}), denoted by $h_{\rm
top}$.  We denote by
$\mathcal{M}_{\rm inv}$ the set of all $T$-invariant probability
Borel measures on $X$ and by $\mathcal{M}_{\rm erg}$ its subset of
all ergodic measures. The measure-theoretic entropy of $\mu$ in
$\mathcal{M}_{\rm inv} $ is denoted by $h_\mu$.

For $\mu \in \mathcal{M}_{\rm inv}$,  the set $G_{\mu}$ of {\em
$\mu$-generic points}  is defined by
$$
   G_{\mu}:=\left\{x \in X: \frac{1}{n}\sum_{j=0}^{n-1} \delta_{T^j x}
   \stackrel{w^*}{\longrightarrow} \mu \right\},
$$
where $\stackrel{w^*}{\longrightarrow}$ stands for the weak star
convergence of  measures.
Bowen (\cite{Bowen}) proved that on any dynamical system, we
have $h_{\rm top} (G_\mu)\le h_\mu$ for any $\mu \in
\mathcal{M}_{\rm inv}$. For ergodic measure measure $\mu$, we get equality. But in general,
the equality doesn't hold.
A dynamical system $(X, T)$ is said to be {\em saturated} if for
any $\mu \in \mathcal{M}_{\rm inv}$, we have $h_{\rm
top}(G_{\mu})=h_{\mu}$.
It is proved in \cite{FLP} that
 systems of specification are saturated.

In this paper, we shall prove a variational principle which relates the topological entropy  $h_{\rm
top}(E_\Phi(\alpha))$ to the measure theoretic entropies of invariant measures in the following set, called
$(\Phi, \alpha)$-fiber,
$$
    \mathcal{M}_\Phi(\alpha) =\left\{\mu \in \mathcal{M}_{\rm inv} : \int_{X^r} \Phi d \mu^{\otimes r}=\alpha \right\}
$$
where $\mu^{\otimes r} =\mu \times \cdots \times \mu$ is the product of $r$ copies of $\mu$.

\begin{theorem}\label{VP}
Suppose that the dynamical system $(X, T)$ is saturated. Let $\Phi \in C(X^r)$ ($r\ge 1$). 
If  $
\mathcal{M}_\Phi(\alpha)= \emptyset$, we have $E_\Phi(\alpha)=\emptyset$.
If  $ \mathcal{M}_\Phi(\alpha)\not=
\emptyset$, we have
\begin{equation}\label{variational-principle}
         h_{\rm top} (E_\Phi(\alpha)) =
        \sup_{\mu \in \mathcal{M}_\Phi(\alpha)} h_\mu.
\end{equation}
\end{theorem}

Theorem~\ref{VP} is well known when $r=1$ (see e.g. \cite{FFW,FLP,BSS,Barreira}).
In particular, it is known that  for regular potential $\Phi$, $\alpha \mapsto  h_{\rm top} (E_\Phi(\alpha))$
is an analytic function (see e.g. \cite{Fan1994,Ruelle}). But as we shall see, when $r\geq 2$, this function can admit discontinuity even for "very regular"
potentials.

The above consideration was motivated by the following problem. Recently the multiple ergodic limit 
\begin{equation}\label{MEA}
      M_{\Phi}(x):=
\lim_{n\to\infty}\frac{1}{n}\sum_{i=0}^{n-1}\Phi(\sigma^ix,\sigma^{2i}x,\cdots,T^{ri}x)
\end{equation}
have been studied by Furstenberg (\cite{Furstenberg}), Bergelson (\cite{Bergelson}), Bourgain (\cite{Bourgain}), Assani (\cite{Assani}), Host and Kra (\cite{HK}), and others. Fan, Liao and Ma proposed in \cite{FLM} to give a multifractal analysis of the multiple ergodic average $M_{\Phi}$, in other words, to determine the Hausdorff dimensions of the level sets
$$L_{\Phi}(\alpha)=\{x\in X: M_{\Phi}(x)=\alpha \}.$$
This problem in its generality remains open.

However, there are two results for the shift dynamics on symbolic space and for some special potentials $\Phi$. The first one concerns the case where $X=\{-1,1\}^\N$, $T$ is the shift and $\Phi(x_1,\cdots,x_r)=x_1^{(1)}\cdots x_r^{(1)}$ ($x_i^{(1)}$ being the first coordinate of $x_i$). By using Riesz products, the authors in \cite{FLM} proved that for $\alpha\in [-1,1]$ we have
$$\dim L_{\Phi}(\alpha)=1-\frac{1}{r}-\frac{1}{r}\left(\frac{1-\alpha}{2}\log_2\frac{1-\alpha}{2}+\frac{1+\alpha}{2}\log_2\frac{1+\alpha}{2}\right).$$
The second one concerns the case where $X=\{0,1\}^\N$, $T$ is the shift and $\Phi(x_1,x_2)=F(x_1^{(1)},x_2^{(1)})$ is a function depending only on the first coordinates $x_1^{(1)}$ and $x_2^{(1)}$ of $x_1$ and $x_2$. The multifractal analysis of these double ergodic average was determined in \cite{FSW}. A related work was done in \cite{KPS} answering a question in \cite{FLM} about the Hausdorff dimension of a subset of $L_{\Phi}(\alpha)$  for extremal values of $\alpha$. 

As shown in \cite{FSW}, the dimension of the ``mixing part'' of $L_{\Phi}(\alpha)$ which is defined by 
$$\sup \left\{\dim \mu : \mu(L_{\Phi}(\alpha))=1,\  \mu \ {\rm is\ mixing}\right\}$$
is equal to 
$$\sup\left\{\dim \mu : \int \Phi d\mu^{\otimes r}=\alpha,\ \mu\ {\rm is \ mixing}\right\}.$$
This equality is very similar to the variational principal stated in Theorem \ref{VP}.

In Section \ref{V-stat}, we recall some facts about V-statistics. In Section \ref{Top Ent}, we recall some notions like topological entropy, generic points and specification property. The main theorem, Theorem \ref{VP}, is proved in Section \ref{proof}. In Section \ref{examples}, we examine the special case of full shift together with some examples. We will see that, even for very regular function $\Phi$, the function $\alpha\to h_{\rm top} (L_\Phi(\alpha))$ may admit discontinuity. 

To finish this introduction, we emphasise that the problem of multifractal analysis of multiple ergodic limits remains largely open.

\section{V-statistics}\label{V-stat}
V-statistics are tightly related to  U-statistics which are well known in statistics.
Let $\mu$ be a probability law  on $\mathbb{R}$. A U-parameter of $\mu$ is defined through
a function called  kernel $h: \mathbb{R}^d \to \mathbb{R}$ by
$$
      \theta(\mu) = \theta_h(\mu) = \int_{\mathbb{R}^d} h d \mu^{\otimes d}
$$
 where $\mu^{\otimes d}$ is the product measure
$\mu \times \cdots \times \mu$ ($d$ times) on $\mathbb{R}^d$.
This $U$-statistics is well defined for all $\mu$ such that the integral exists.

In statistics, U-parameters are also
called estimable parameters and they constitute the set of all parameters that can be estimated in an unbiased
fashion. A fundamental problem in statistics is the estimation of a parameter $\theta(\mu)$
for an unknown probability law $\mu$. To estimate a U-parameter $\theta_h$, people employ the
U-statistics for $\theta_h$:
$$
    U_h(X_1, \cdots, X_n) =\frac{(n-d)!}{n!} \sum h(X_{i_1}, \cdots, X_{i_d})
$$
where the sum is taken over all $(i_1, \cdots, i_d)$ with $i_j$'s distinct and $1\le i_j\le n$,
where $X_1, \cdots, X_d$ is a sequence of observations of $\mu$. Closely related to U-statistics
is the V-statistics (von Mises statistics):
$$
     V_h(X_1, \cdots, X_n) = n^{-d}\sum_{1\le i_1, \cdots, i_d \le n} h(X_{i_1}, \cdots, X_{i_d}).
$$

People expect that $U_h(X_1, \cdots, X_n)$ converges almost surely to $\theta_h$.
This fact, if it holds, allows one to estimate $\theta_h$ using observations.
If it is the case, we say the U-parameter strong law of large numbers (SLLN) holds.
The U-statistics SLLN had been well studied for  independent observations. In \cite{ABDGHW},
the authors have studied the  U-statistics SLLN for ergodic stationary process
$(X_n)$, i.e. $X_n = f\circ T^n$ where $T$ is ergodic measure-preserving transformation
on a probability space $(\Omega, \mathcal{A}, \mathbb{P})$, $f: \Omega \to \mathbb{R}$ is a measurable function
and $X_1$ admits $\mu$ as probability law.

If $h$ is a kernel bounded by a integrable function and if $(X_n)$ is ergodic, it can be proved (see \cite{ABDGHW}) that
almost surely
$$
     \lim_{n\to \infty}|U_h(X_1, \cdots, X_n)- V_h(X_1, \cdots, X_n)|=0.
$$
It is also proved in \cite{ABDGHW} that the U-statistics SLLN holds if the kernel $h$ is continuous.
In the following, we consider only V-statistics.

\section{Topological entropy}\label{Top Ent}
For any
integer $n\ge 1$,  the Bowen metric $d_n$ on $X$ is defined by
$$
      d_n(x, y) = \max_{0\le j <n} d(T^jx, T^j y).
$$
For any $\epsilon >0$, we will denote by $B_n(x, \epsilon)$ the
open $d_n$-ball centered at $x$ of radius $\epsilon$.

Let $ Z \subset X$ be a subset of $X$. Let $\epsilon >0$. A cover is a
collection of Bowen balls (at most countable) $R=\{B_{n_i}(x_i,
\epsilon)\}$ such that $Z \subset \bigcup_i B_{n_i}(x_i,
\epsilon)$.  For such a cover $R$, we put $n(R) = \min_i
n_i$.  Let $s\ge 0$. Define
$$
     H^s_n (Z, \epsilon) = \inf_R \sum_i \exp(-s n_i),
$$
where the infimum is taken over all covers $R$ of $Z$ with $n(R)
\ge n$.  The quantity $H^s_n (Z, \epsilon)$ being a non-decreasing
function of $n$, the following limit exists
$$
   H^s (Z, \epsilon)  = \lim_{n \to \infty} H^s_n (Z, \epsilon).
$$
Consider the quantity $H^s (Z, \epsilon)$  as a function of
$s$, there exists a critical value, which we denote by $h_{\rm
top} (Z, \epsilon)$, such that
$$
     H^s (Z, \epsilon) =\left\{ \begin{array} {ll} +\infty, & s <
                        h_{\rm top} (Z, \epsilon) \\ 0 , & s> h_{\rm
                        top} (Z, \epsilon).  \end{array} \right.
$$
The following limit exists
$$
          h_{\rm top} (Z) = \lim_{\epsilon \to 0} h_{\rm top} (Z,
          \epsilon).
$$
The limit $h_{\rm top} (Z)$ is called the {\em topological
entropy} of $Z$ (\cite{Bowen}).

For $x \in X$, we denote by $V(x)$ the set of all weak limits of
the sequence of probability measures $n^{-1}\sum_{j=0}^{n-1}
\delta_{T^j x}$. Recall that $X$ is compact.    It is clear then that for any $x$ we have
 $$ \emptyset \not= V(x)
\subset \mathcal{M}_{\rm inv}.$$
 The following lemma is due to Bowen (\cite{Bowen}).

\begin{lemma}\label{Bowen}
    For $t\ge 0$, we have $h_{\rm top} (B^{(t)})\le t$ where
    $$
         B^{(t)} =\left\{x\in X:  \exists\  \mu \in V(x) \
         \mbox{\rm satisfying}\ h_\mu \le t\right\}.
    $$
\end{lemma}

The set $G_\mu$ of $\mu$-generic points is the set of all $x$ such that
 $V(x)=\{\mu\}$.
The Bowen lemma implies that
$$h_{\rm
top}(G_\mu)\le h_\mu$$
for any invariant measure $\mu$. It is simply because
 $x \in G_\mu $ implies $\mu \in V(x)$. Bowen also proved that
the inequality becomes equality when $\mu$ is ergodic. However, in
general, we do not have the equality  and it is even possible that $G_\mu =
\emptyset$. In fact, $\mu(G_\mu) = 1 \ {\rm or}\ 0$ according to whether
$\mu$ is ergodic or not (see \cite{DGS}).

The equality $h_{\rm
top}(G_\mu)= h_\mu$ does hold for any invariant probability measure in  any dynamical
system with specification (\cite{FLP}).

\begin{lemma}\label{FLP}
    Any dynamical system with specification $(X, T)$ is saturated. In other words,
    $h_{\rm
top}(G_\mu)= h_\mu$ for any $\mu \in \mathcal{M}_{\rm inv}$.
\end{lemma}

A dynamical system $(X,T)$ is said to satisfy the {\em
specification property} if for any $\epsilon >0$ there exists an integer
$m(\epsilon)\ge 1$ having the property that  for any integer $k\ge2$,
for any $k$ points $x_1,\ldots,x_k$ in $X$, and for any integers
$$
  a_1 \le b_1 < a_2 \le b_2 < \cdots < a_k \le b_k
$$
with $
    a_i - b_{i-1} \ge m(\epsilon) \quad (\forall 2 \le i \le k),
$ there exists a point $y\in X$ such that
$$
   d(T^{a_i+n} y,T^n x_i) < \epsilon
\qquad (\forall
   \ 0 \le n \le b_i-a_i,  \quad \forall 1 \le i \le k).
$$

The specification
property implies the topological mixing. Blokh (\cite{Blokh})
proved that these two properties are equivalent for continuous
interval transformations.
Mixing subshifts of finite type satisfy the specification
property. In general, a subshift satisfies the specification if
for any admissible words $u$ and $v$ there exists a word $w$ with
$|w|\le k$ (some constant $k$) such that $uwv$ is admissible. For
$\beta$-shifts defined by $T_\beta x = \beta x (\!\!\mod 1)$,
there is only a countable number of $\beta$'s such that the
$\beta$ shifts admit Markov partition (i.e. subshifts of finite
type), but an uncountable number of $\beta$'s such that the
$\beta$-shifts satisfy the specification property
(\cite{Schmel97}).

We finish this section by mentioning that continuous functions on $X^r$
can be uniformly approximated by tensor functions.  It is
a consequence of the Stone-Weierstrass theorem.

\begin{lemma}\label{SW}
    Let $F \in C(X^r)$. For any $\epsilon>0$, there exists a function of the form
    $$
        \widetilde{F}(x_1, \cdots, x_r) = \sum_{j=1}^n f_j^{(1)}(x_1)f_j^{(2)}(x_2) \cdots f_j^{(r)}(x_r)
    $$
    where $f_j^{(i)} \in C(X)$, such that $\|F-\widetilde{F}\|_\infty<\epsilon$.
\end{lemma}

We will write
$$
        \widetilde{F} = \sum_{j=1}^n f_j^{(1)}\otimes f_j^{(2)}\otimes \cdots \otimes f_j^{(r)}.
    $$

\section{Proof of Theorem~\ref{VP}}\label{proof}

We can actually consider Banach-valued $V$-statistics.
More than Theorem~\ref{VP} can be proved.

Let $\mathbb{B}$ be a real Banach space and
$\mathbb{B}^*$ its dual space. The duality will be denoted by
$\langle y, x\rangle$ ($x \in \mathbb{B}, y\in \mathbb{B}^*$). We consider $\mathbb{B}^*$ as a
locally convex topological space with the weak star topology
$\sigma(\mathbb{B}^*, \mathbb{B})$. For any $\mathbb{B}^*$-valued
continuous function $\Phi: X \to \mathbb{B}^*$, we consider its
$V$-statistics $V_\Phi(n, x)$ as before, formally in the same way.

 Fix a subset
$W \subset \mathbb{B}$. For a sequence $\{\xi_n\} \subset
\mathbb{B}^*$ and a point $\xi \in \mathbb{B}^*$, we denote by
$\limsup_{n \to \infty} \xi_n \stackrel{W}{\le} \xi$ the fact
$$
\limsup_{n \to \infty} \langle \xi_n , w\rangle \le \langle  \xi,
w\rangle\ {\rm \ for \ all\ } w \in W.
$$
 It is clear that
$\limsup_{n \to \infty} \xi_n \stackrel{\mathbb{B}}{\le} \xi$
 means $\xi_n$ converges to $\xi$ in the weak star topology
$\sigma(\mathbb{B}^*, \mathbb{B})$.

Given $\alpha \in \mathbb{B}^*$ and $W \subset \mathbb{B}$. We define
$$
E_\Phi(\alpha, W) = \left\{ x \in X:  \limsup_{n \to \infty}
V_\Phi(n, x) \stackrel{W}{\le}
  \alpha
\right\}
$$
$$
       \mathcal{M}_\Phi(\alpha, W) = \left\{\mu\in \mathcal{M}_{\rm
       inv}: \int \Phi d\mu \stackrel{W}{\le} \alpha \right\}
$$
where $\int \Phi d \mu$ denotes the vector-valued integral in
Pettis' sense (see \cite{Rudin}) and the inequality
``$\stackrel{W}{\le}$" means
$
             \int \langle \Phi, w \rangle d\mu \le
             \langle \alpha, w \rangle
             \ \ \mbox{\rm for \ all }\ w \in W.
$

\begin{theorem}\label{VP2}
Suppose that the dynamical system $(X, T)$ is saturated.  If  $
\mathcal{M}_\Phi(\alpha, W)= \emptyset$, we have $E_\Phi(\alpha, W)=\emptyset$. If  $ \mathcal{M}_\Phi(\alpha, W)\not=
\emptyset$, we have
\begin{equation}\label{variational-principle}
         h_{\rm top} (E_\Phi(\alpha, W)) =
        \sup_{\mu \in \mathcal{M}_\Phi(\alpha, W)} h_\mu.
\end{equation}
\end{theorem}

\begin{proof}
We prove the first assertion by showing  that
$E_\Phi(\alpha, W)\not=\emptyset$  implies $\mathcal{M}_\Phi(\alpha, W)\not= \emptyset$.
Let $x\in E_\Phi(\alpha, W)$. There exists a measure $\mu \in V(x)\subset \mathcal{M}_{\rm inv}$ and a sequence of integers $(n_k)$ such that
\begin{equation}\label{limit_mu}
  \mu = w^*\!-\!\lim_{k \to \infty} \frac{1}{n_k}\sum_{j=1}^{n_k} \delta_{T^j x}.
\end{equation}
We are going to show that $\mu \in \mathcal{M}_\Phi(\alpha, W)$.

Let $w\in W$. Then $\langle \Phi, w \rangle$ is a continuous function on $X$.
For an arbitrarily small $\epsilon >0$, by the Stone-Weierstrass  theorem (See Lemma\ref{SW}) there exists a function $\widetilde{\Phi}$ of the form
$$
   \widetilde{\Phi}= \sum_j f_j^{(1)} \otimes f_j^{(2)} \otimes\cdots \otimes f_j^{(r)}
$$
(finite sum of tensor products) such that $$
\|\langle \Phi, w \rangle  - \widetilde{\Phi}\|_\infty\le \epsilon.
$$ Notice that
$$
  V_{\widetilde{\Phi}}(n, x) = \sum_j \prod_{i=1}^r \frac{S_n f_j^{(i)}(x)}{n}
$$
where
$$S_n f(x)= \sum_{k=1}^n f(T^k x)$$
 denotes the ergodic sum  for a given function $f$. According to
(\ref{limit_mu}), we have
\begin{equation}\label{limit_Tphi}
  \lim_{k \to \infty} V_{\widetilde{\Phi}}(n_k, x)=
  \sum_j \prod_{i=1}^r \int_X f_j^{(i)}d \mu = \int_{X^r} \widetilde{\Phi} d\mu^{\otimes d}.
\end{equation}
On the other hand, we  write $$
\int \langle \Phi, w\rangle d \mu^{\otimes d} - \langle \alpha, w\rangle
= \sigma_1+ \sigma_2 + \sigma_3 + \sigma_4
$$
where
\begin{eqnarray*}
 \sigma_1 & = & \int (\langle \Phi, w\rangle-\widetilde{\Phi}) d\mu^{\otimes d} \\
 \sigma_2 & =&  \int \widetilde{\Phi} d\mu^{\otimes d}- V_{\widetilde{\Phi}}(n_k, x)\\
 \sigma_3 & = & V_{ \widetilde{\Phi}}(n_k, x)- V_{\langle\Phi, w\rangle}(n_k, x)\\
 \sigma_4 &=&  V_{\langle\Phi, w\rangle} (n_k, x)-\langle  \alpha, w\rangle.
\end{eqnarray*}
We have
   $$
      |\sigma_1| \le \epsilon, \quad |\sigma_3| \le \epsilon, \quad
      \lim \sigma_2=0, \quad \limsup \sigma_4 \le 0.
   $$
So, we get
 $$
 \int \langle \Phi, w\rangle d \mu^{\otimes d} \le \langle \alpha, w\rangle+2 \epsilon.
 $$
Since
 $\epsilon$ is arbitrary,
we have thus proved that $\mu \in \mathcal{M}_{\rm inv}(\alpha, W)$.
The first assertion is then proved.

Prove now the second assertion.
What we have just proved also implies
$$
   E_\Phi(\alpha, W) \subset B^{(t)}=\{x \in X: \exists \mu \in V(x) \ {\rm such \ that} \ h_\mu \le t\}
$$
where $t = \sup_{\mu \in \mathcal{M}_\Phi(\alpha, W)} h_\mu$. By the Bowen lemma (Lemma~\ref{Bowen}), we get
\begin{equation}\label{upperbound}
  h_{\rm top}(E_\Phi(\alpha))
  \le
  \sup_{\mu \in \mathcal{M}_\Phi(\alpha)} h_\mu.
\end{equation}
To finish the proof of the second assertion, it suffices to prove the reverse inequality of (\ref{upperbound}).
Let $\mu \in \mathcal{M}_\Phi(\alpha)$. Let $x\in G_\mu$.  For any $\epsilon >0$ and any $w\in W$,
consider $\widetilde{\Phi}$ as above. We have
$$
    \lim_{n \to \infty} V_{\widetilde{\Phi}}(n, x) = \int_{X^r} \widetilde{\Phi} d \mu^{\otimes r}.
$$
It follows that
\begin{eqnarray*}
    \limsup_{n \to \infty} V_{\langle \Phi, w\rangle }(n, x)
    & \le &  \lim_{n \to \infty} V_{\widetilde{\Phi}}(n, x) +\epsilon\\
    & = &  \int_{X^r} \widetilde{\Phi} d \mu^{\otimes r} +\epsilon \\
     & \le  & \int_{X^r} \langle \Phi, w\rangle  d \mu^{\otimes r} +2\epsilon\le \langle \alpha, w\rangle
     +2\epsilon.
\end{eqnarray*}
Letting $\epsilon \to 0$ we get
$$
      \limsup_{n \to \infty} \langle V_{{\Phi}}(n, x), w\rangle \le  \langle \alpha, w\rangle.
$$
In other words, we have proved $G_\mu \subset E_\Phi(\alpha, W)$ for all $\mu \in \mathcal{M}_{\rm inv}(\alpha, W)$.
 So,
$$h_{\rm top}(E_\Phi(\alpha))\ge h_{\rm top} (G_\mu).$$
By Lemma~\ref{FLP}, $h_{\rm top} (G_\mu)=h_\mu$.
Taking the supremum over $\mu \in \mathcal{M}_{\rm inv}(\alpha, W)$ leads to the reverse inequality of (\ref{upperbound}).

\end{proof}

\section{Example: Shift dynamics}\label{examples}

Let $(X, T) = (\Sigma_m, \sigma)$ with $m\ge 2$, where
 $\sigma\colon\Sigma_m\to \Sigma_m$ is the shift on the space $\Sigma_m=\{0,1, \cdots, m-1\}^\N$.

 Let
 $$L_{(\Phi, W)}=\{ \alpha \in \mathbb{B}^*: E_\Phi(\alpha, W)\not= \emptyset\}.
 $$
 If $W=\mathbb{B}$, we write $L_{\Phi}=L_{(\Phi, W)}$. Define $f_{(\Phi, W)}: L_{(\Phi, W)}\to \mathbb{R}$ by
 $$
     f_{(\Phi, W)}(\alpha)= h_{\rm top} (E_\Phi(\alpha, W)).
 $$

\begin{theorem}\label{USC}
$f_{(\Phi, W)}: L_{(\Phi, W)}\to \mathbb{R}$ is upper semi-continuous.

\end{theorem}
\begin{proof}
    Let $\alpha_n, \alpha \in L_{(\Phi, W)}$. Suppose $\alpha_n \to \alpha$
   in the weak star topology.
   We have to show that
   $$
      \limsup_n f_{(\Phi, W)}(\alpha_n)\le f_{(\Phi, W)}(\alpha).
   $$
   Since each fiber like $\mathcal{M}_{\rm inv}(\alpha, W)$ is compact, there are maximizing measures
   $\mu_{\alpha_n}\in \mathcal{M}_{\rm inv}(\alpha_n, W)$ and $\mu_\alpha\in \mathcal{M}_{\rm inv}(\alpha, W)$ such that
   \begin{equation}\label{USC1}
         f_{(\Phi, W)}(\alpha_n)= h_{\alpha_n}, \qquad f_{(\Phi, W)}(\alpha)= h_{\alpha}.
  \end{equation}
   Without loss of generality, we can assume that $\mu_{\alpha_n}$  converge weakly, say to $\mu^*$.
   Since
   $$
      \forall w\in W,    \quad \int \langle\Phi, w \rangle d \mu_n \le \langle \alpha_n, w \rangle,
   $$
   taking limit shows that $\mu^* \in \mathcal{M}_{\rm inv}(\alpha, W)$. It follows that
     \begin{equation}\label{USC2}
     h_{\mu^*}\le h_{\mu_\alpha}.
       \end{equation}
   On the other hand, recall that for the shift dynamics, the entropy function $\mu \mapsto h_\mu$
   is upper semi-continuous. So,
   \begin{equation}\label{USC3}
       \limsup_n h_{\alpha_n} \le h_{\mu^*}.
     \end{equation}
  We combine (\ref{USC1}),(\ref{USC2}) and (\ref{USC3}) to finish the proof.
\end{proof}

\begin{theorem}\label{USC} Assume that $\Phi$ is a  function defined on
$\Sigma_m^r$ ($r\ge 1$) which depends only on the first $k$ coordinates of each of its variables
($k\ge 1$). Then the suppremum in the variational principle (\ref{variational-principle}) is attained by a
$(k-1)$-Markov measure.
\end{theorem}

\begin{proof} This is just because
the integral $\int \Phi d\mu^{\otimes r}$ depends only on the values
$\mu([a_1, \cdots,a_k])$ of the measure $\mu$ on cylinders $[a_1, \cdots, a_k]$
and there exists a $(k-1)$-Markov measure $\nu$ such that
$$
    \mu([a_1, \cdots,a_k]) = \nu([a_1, \cdots,a_k])
$$
for all cylinders $[a_1, \cdots, a_k]$ and such that $h_\nu \ge h_\mu$.
\end{proof}

In particular, if $k=1$, maximizing measures are Bernoulli measures.
For the Bernoulli measure $\mu_p$ determined by a probability vector $p=(p_0, \cdots, p_{m-1})$,
we have $h_{\mu_p}=H_1(p)$ where
$$
      H_1(p) = -\sum_{j=0}^{m-1} p_j \log p_j.
$$
Suppose that the function $\Phi$ is a product of $r$ functions and each of its factor depends only on the first coordinate, i.e.
 $$
 \Phi(x^{(1)}, \cdots, x^{(r)})
= \phi_1(x^{(1)}_1)\cdots \phi_r(x^{(r)}_1).
$$
Let $$
    A(p)= \int_{\Sigma_m^r} \Phi(x^{(1)}, \cdots, x^{(r)})  d\mu_p(x^{(1)})\cdots d\mu_p(x^{(r)}).
$$
 Notice that $E_\Phi(\alpha)\not=\emptyset$ iff $\alpha=A(p)$ for some  probability vector $p=(p_0, \cdots, p_{m-1})$.
 The following result is a direct consequence of the last theorem.

 \begin{theorem}\label{k=1} Let $\Phi(x^{(1)}, \cdots, x^{(r)})
= \phi_1(x^{(1)}_1)\cdots \phi_r(x^{(r)}_1)$. We have
$$
      A(p) = \prod_{k=1}^r \sum_{j=0}^{m-1} \phi_k(j) p_j.
$$
For any $\alpha$ satisfying $E_\Phi(\alpha)\not=\emptyset$, we have
 \begin{equation}
      h_{\rm top}(E_\Phi(\alpha)) = \max_{A(p)=\alpha}  H_1(p)
  \end{equation}
  where the maximum is taken over all probability vectors $p$ satisfying $A(p)=\alpha$.
\end{theorem}

If $k=2$, maximizing measures are Markov measures. A Markov measure $\mu_{p, P}$ is determined by
a probability vector $p$ and a transition matrix $P$.
Its entropy is equal to
$$
      H_2(p, P) = -\sum_{i=0}^{m-1} p_i \sum_{j=0}^{m-1} p_{i,j} \log p_{i,j}.
$$
Suppose $\Phi(x^{(1)}, \cdots, x^{(r)})$ is of the form
 $
 \phi_1(x^{(1)}_1, x^{(1)}_2)\cdots \phi_r(x^{(r)}_1, x^{(r)}_2).
$  
Let $$
    A(p, P)= \int_{\Sigma_m^r} \Phi(x^{(1)}, \cdots, x^{(r)})  d\mu_{p,P}(x^{(1)})\cdots d\mu_{p,P}(x^{(r)}).
$$

\begin{theorem}\label{k=2} Let $\Phi(x^{(1)}, \cdots, x^{(r)})
= \phi_1(x^{(1)}_1, x^{(1)}_2)\cdots \phi_r(x^{(r)}_1, x^{(r)}_2)$. We have
$$
      A(p,P) = \prod_{k=1}^r \sum_{i, j=0}^{m-1} \phi_k(i,j) p_ip_{i,j}.
$$
For any $\alpha$ satisfying $E_\Phi(\alpha)\not=\emptyset$, we have
 \begin{equation}
      h_{\rm top}(E_\Phi(\alpha)) = \max_{A(p, P)=\alpha}  H_2(p, P)
  \end{equation}
  where the maximum is taken over all couples $(p, P)$ satisfying $A(p, P)=\alpha$.
\end{theorem}

Let us consider two examples. We will use the following trivial property of the entropy
function $H(x)=-x\log x-(1-x)\log(1-x)$.

\begin{lemma}
Given two numbers $p_1, p_2 \in [0,1]$. We have
$$
H(p_1)<H(p_2)\ \  \mbox{\rm iff} \ \ |p_1-1/2|> |p_2-1/2|.
$$
We have $
H(p_1)= H(p_2)$  iff  $ |p_1-1/2|= |p_2-1/2|$.
\end{lemma}
\medskip

{\em Example 1.} Consider the case $m=2$, $k=1$ and $r=2$. Let
$x=p_1$. Then $p_0=1-x$ and we have
$$
     A(p) = [\phi_1(0)(1-x) +\phi_1(1)x][\phi_2(0)(1-x) +\phi_2(1)x].
$$
For simplicity, we write $A(x)$ for $A(p)$. Suppose that $\phi_1(0)\neq \phi_1(1)$ and $\phi_2(0)\neq
\phi_2(1)$. Otherwise, the question is trivial. By multiplying
$\phi$ by a constant we can suppose that $A(x)$ is of the form
$$A(x)=(x-a)(x-b).$$
Let $x=x^*$ be the critical point of the quadratic function $A$
(i.e., $x^*=\frac{a+b}{2}$).

Using the last lemma, it is easy to find the unique point $x_\alpha$
such that $$A(x_\alpha)=\alpha, \ \ h_{\rm top
}(E_{\Phi}(\alpha))=H(x_\alpha).$$ 
The point $x_{\alpha}$ is the closest to $1/2$ among those $x$ such that $A(x)=\alpha$.

We distinguish three cases.

\medskip

{\em Case I.} $x^* \le 0$ or $x^*\ge 1$ (see Figure \ref{figure1}).\\
\indent 1. $A(x)$ is strictly monotonic in the interval $[0, 1]$.\\
\indent 2. $L_\Phi$ is the interval with end points $ab$ and $(1-a)(1-b)$. \\
\indent 3. For any $\alpha\in L_\Phi$ , $A(x_\alpha)=\alpha$ admits
a unique solution $x_\alpha$ in $[0, 1]$.

 \medskip

{\em Case II.} $0<x^* \le 1/2$ (see Figure \ref{figure2}).\\
\indent 1. $A(x)$ is strictly monotonic  in the intervals $[x^*, 1]$.\\
\indent 2. $L_\Phi$ is the interval with end points $A(x^*)$ and $(1-a)(1-b)$.\\
\indent 3. For any $\alpha\in L_\Phi$, $A(x_\alpha)=\alpha$ admits a
unique solution $x_\alpha$ in $[x^*, 1]$.

 \medskip

{\em Case III.} $1/2\le x^* <1$ (see Figure \ref{figure3}).\\
\indent 1. $A(x)$ is strictly increasing in the interval $[0, x^*]$.\\
\indent 2. $L_\Phi$ is the interval with end points  $ab$ and  $A(x^*)$.\\
\indent 3. For any $\alpha\in L_\Phi$, $A(x_\alpha)=\alpha$ admits a
unique solution $x_\alpha$ in $[0, x^*]$.

  \begin{figure}
\begin{minipage}[t]{0.55\linewidth}
\centering
\includegraphics[width=6.8cm]{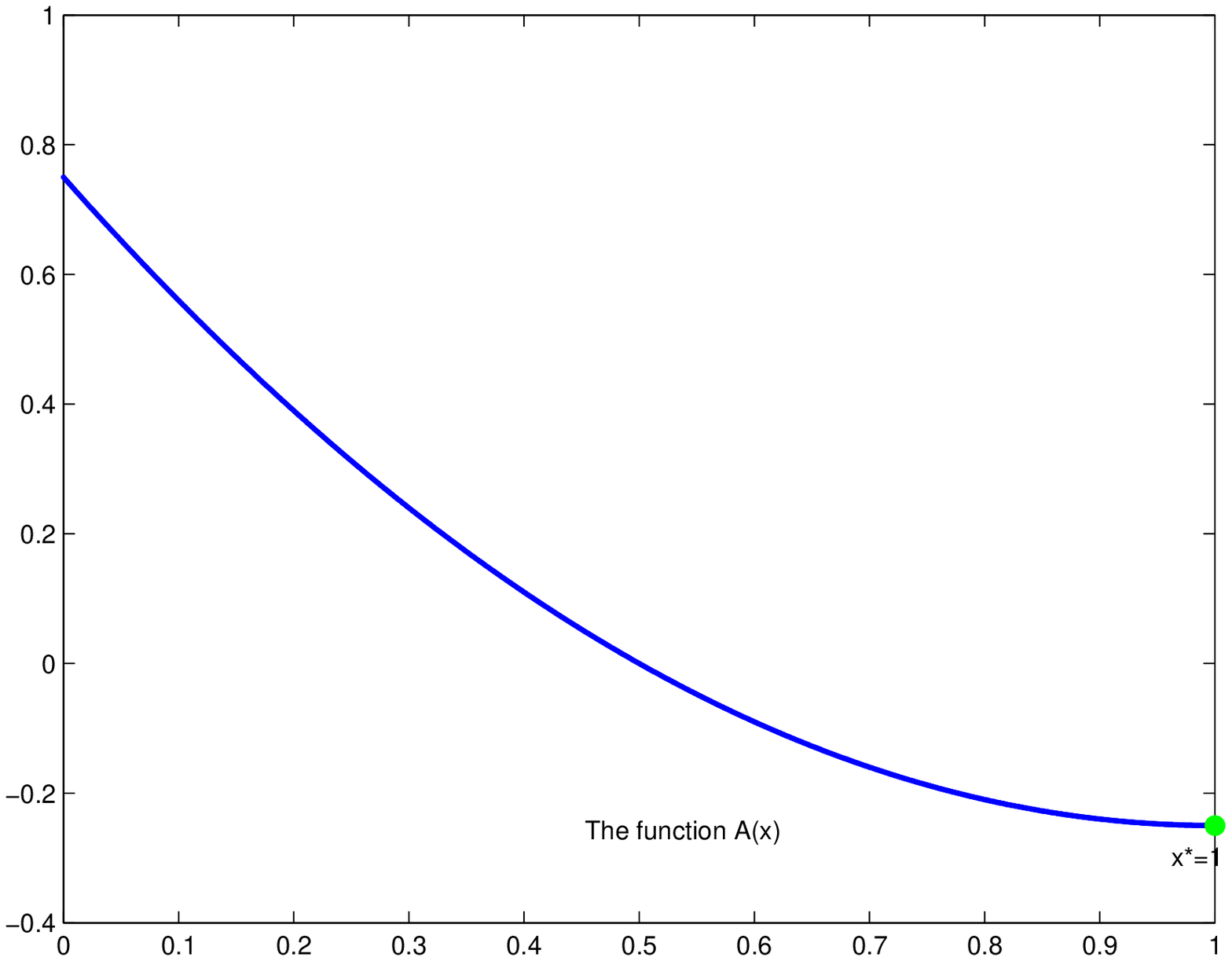}
\end{minipage}%
\begin{minipage}[t]{0.55\linewidth}
\centering
\includegraphics[width=6.8cm]{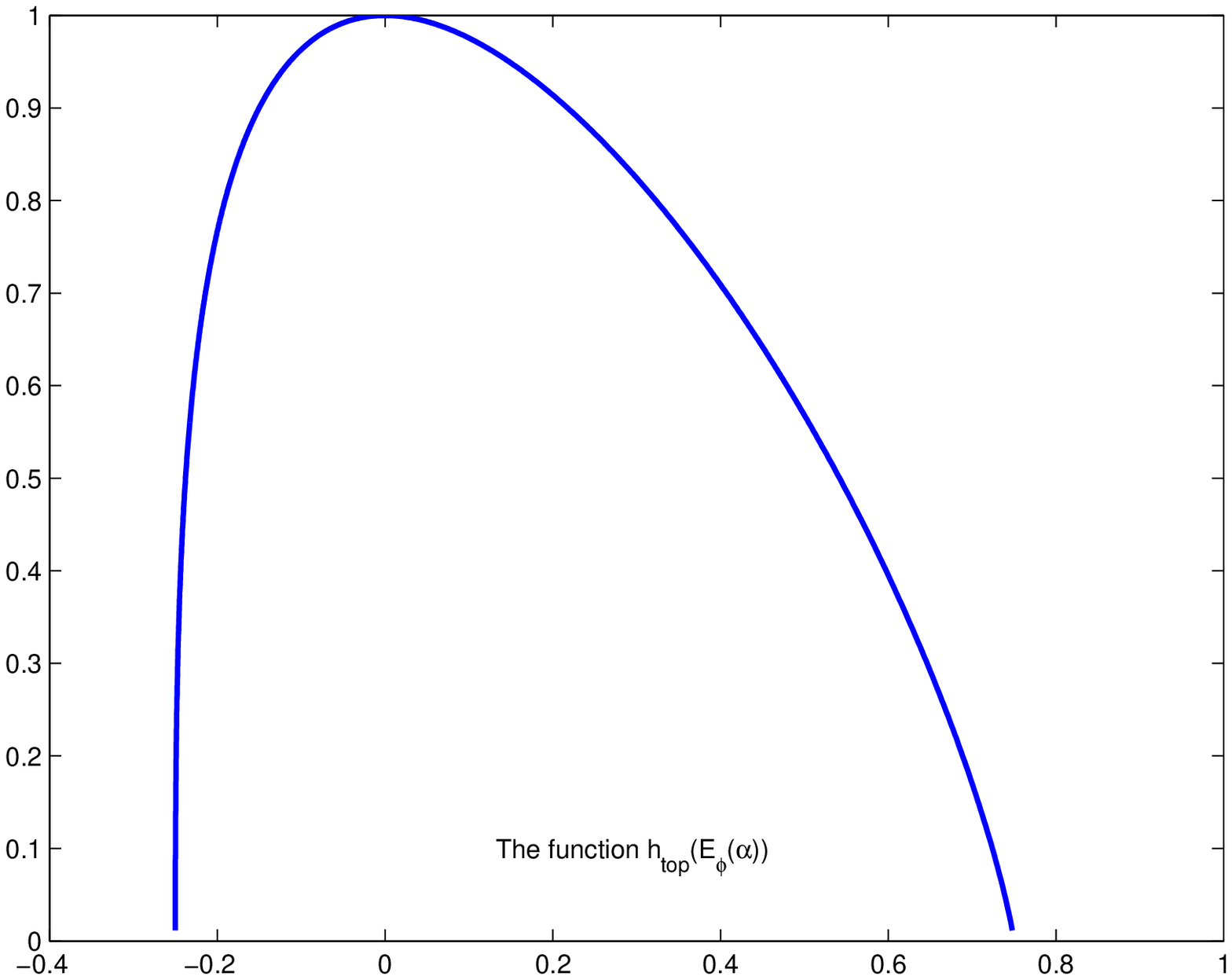}
\end{minipage}
\caption{{\Small Case $x^*=1$ (with $a=0.5$, $b=1.5$)}}
\label{figure1}
\end{figure}

 \begin{figure}
\begin{minipage}[t]{0.55\linewidth}
\centering
\includegraphics[width=6.8cm]{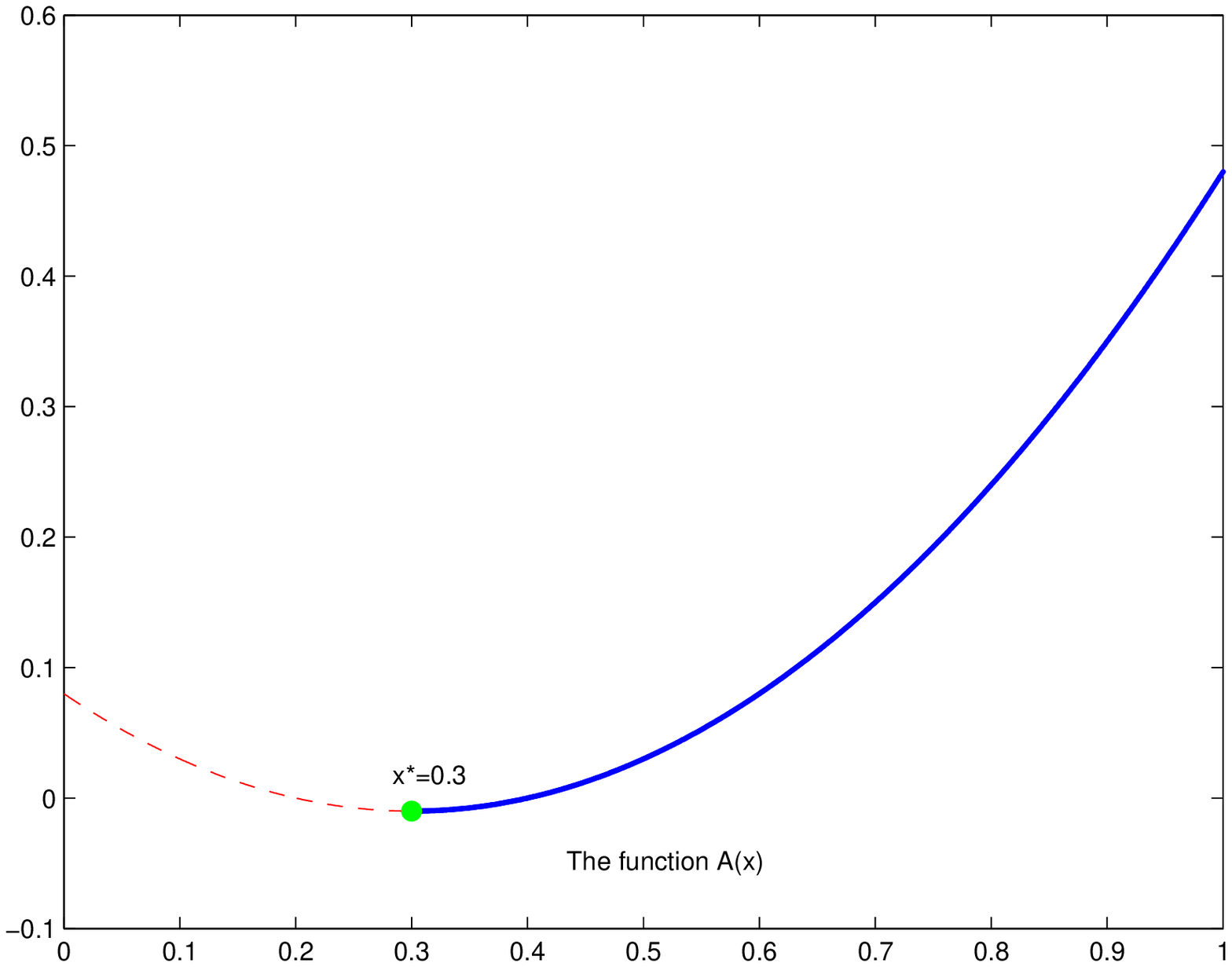}
\end{minipage}%
\begin{minipage}[t]{0.55\linewidth}
\centering
\includegraphics[width=6.8cm]{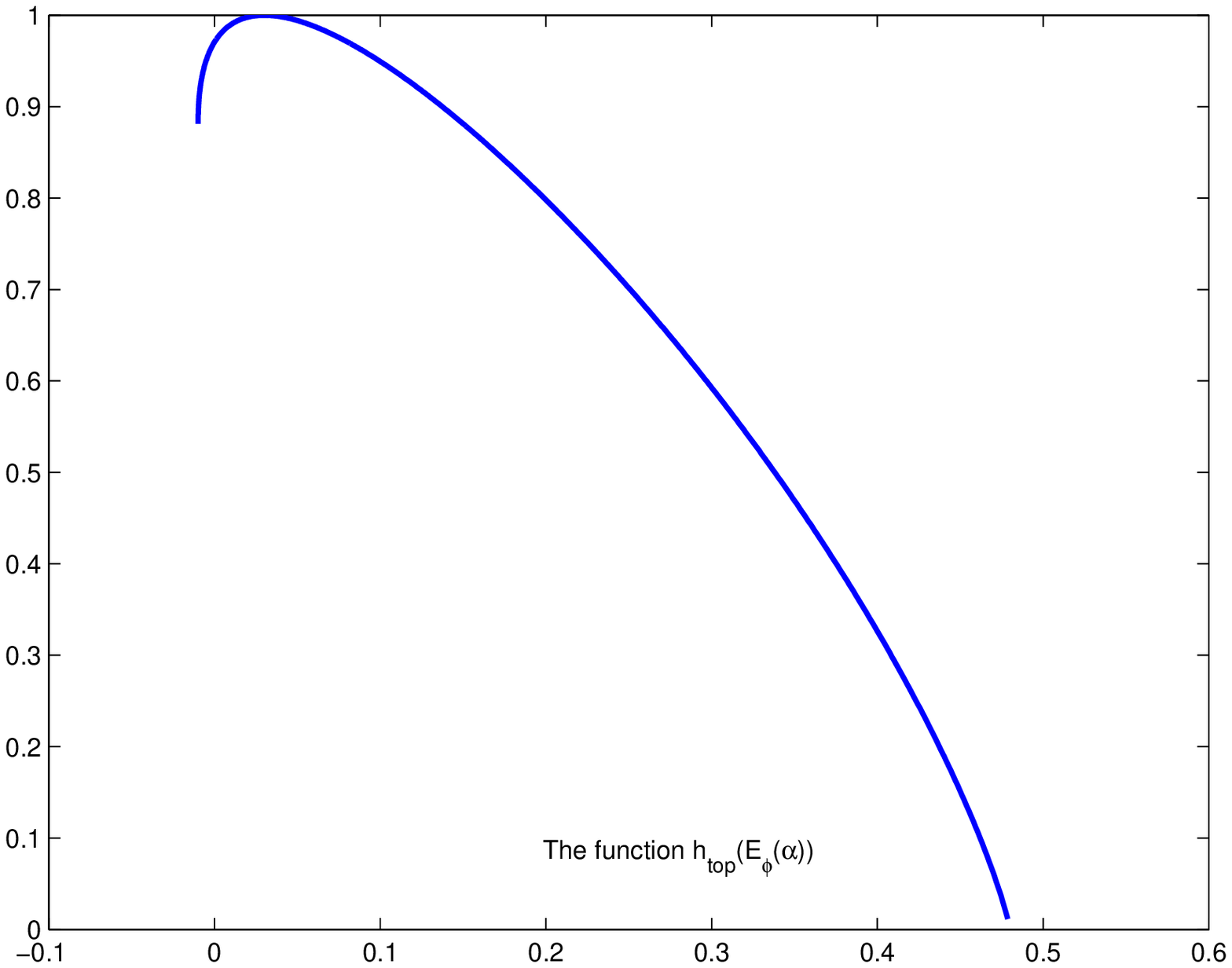}
\end{minipage}
\caption{\Small Case $0<x^*<1/2$ (with $a=0.2, b=0.4$)}
\label{figure2}
\end{figure}

 \begin{figure}
\begin{minipage}[t]{0.55\linewidth}
\centering
\includegraphics[width=6.8cm]{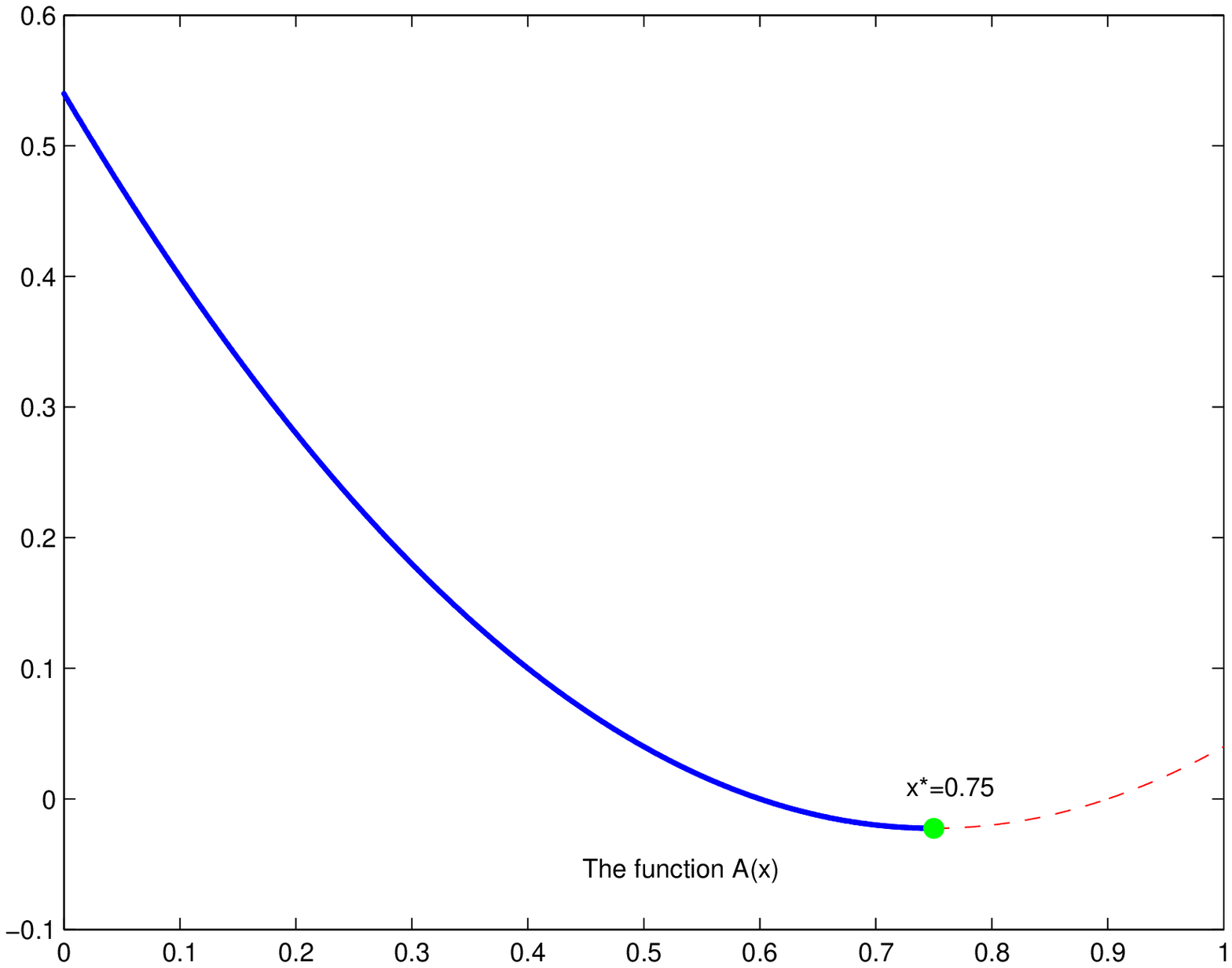}
\end{minipage}%
\begin{minipage}[t]{0.55\linewidth}
\centering
\includegraphics[width=6.8cm]{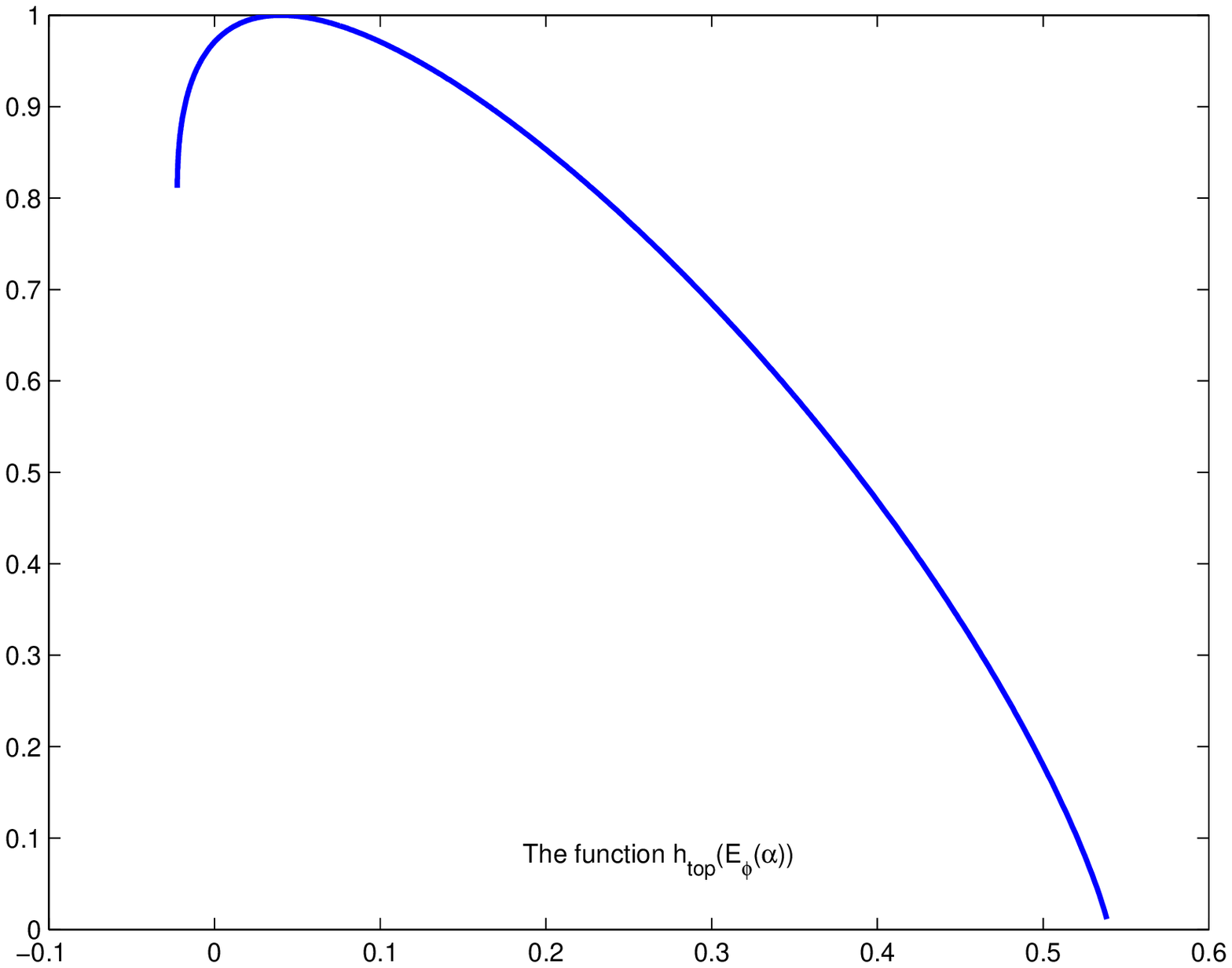}
\end{minipage}
\caption{\Small Case $1/2<x^*<1$ (with $a=0.6, b=0.9$)}
\label{figure3}
\end{figure}

\begin{remark}
We can see in the case $m=2$, $k=1$ and $r=2$ the spectrums are
always continuous (in fact, they are differentiable in the interior of $L_{\Phi}$). In the
following examples we will see that this  is no longer
 the case when $m=2$, $k=1$ and $r=3$.
\end{remark}
\medskip

{\em Example 2.} Consider the case $m=2$, $k=1$ and $r=3$. We have
$$
     A(x) = [\phi_1(0)(1-x) +\phi_1(1)x][\phi_2(0)(1-x) +\phi_2(1)x][\phi_3(0)(1-x)+\phi_3(1)x].
$$

By multiplying $\phi$ by a constant, we can always suppose that $A$ is of
the form
$$A(x)=(x-a)(x-b)(x-c).$$

This cubic polynomial function is either increasing or admit
 a local maximal point $x_{\max}$ and a local minimal point $x_{\min}$ and then we must have
 $x_{\max}<x_{\min}$. As we will see, the continuity of the spectrum depends on the location of $x_{\max}$ and
 $x_{\min}$.

 When $A$ is increasing or when $x_{\max},x_{\min}\notin (0,1)$, $L_\Phi$ is the interval with  $-abc$ and
 $(1-a)(1-b)(1-c)$ as end points. For any $\alpha$ in the interval,
 $A(x_\alpha)=\alpha$ admits a unique solution $x_\alpha$ in $[0,1]$
 and $h_{\rm top}(E_{\Phi}(\alpha))=H(x_\alpha)$. In this
 case the spectrum is continuous (and differentiable).

 Suppose now that $A(x)$ admits a local maximal point $x_{\max}$ and a
 local minimal point $x_{\min}$ (with $x_{\max}<x_{\min}$). Then there
 exist a unique $x'>x_{\min}$ and a unique $x''<x_{\max}$ such that $$A(x')=A(x_{\max}),\ \ A(x'')=A(x_{\min}).$$
 We point out that there are three possible situations: the spectrum
 is continuous, admits one discontinuous point or admits two discontinuous
 points. Before present in detail these three situations we prove the following lemma which will be useful for our discussion.

 \begin{lemma}
 Let $P$ be a polynomial of degree 3 with positive leading
 coefficient. Suppose that $P$ admits a local maximal point
 $x_{\max}$ and a local minimal point $x_{\min}$. Then
 $x_{\max}<x_{\min}$ and
 $$x_1<x_{\max}<x_2,|x_1-x_{\max}|=|x_2-x_{\max}| \Rightarrow P(x_1)<P(x_2)$$
 $$y_1<x_{\min}<y_2,|y_1-x_{\min}|=|y_2-x_{\min}| \Rightarrow P(y_1)<P(y_2)$$
 \end{lemma}

 \begin{proof}
 The fact $x_{\max}<x_{\min}$ follows from $P(-\infty)=-\infty$ and
 $P(+\infty)=+\infty$. By the existence of the extremal points, we
 can write $$P'(x)=\lambda(x-x_{\max})(x-x_{\min})$$ with $\lambda>0$.
 It follows that
 $$ u<x_{\max}<v, x_{\max}-u=v-x_{\max} \Rightarrow \frac{|P'(u)|}{|P'(v)|}=\frac{|u-x_{\min}|}{|v-x_{\min}|}>1.$$
 This means that for two equidistant points from $x_{\max}$, the
 left point climbs quicker than the right point descents. By
 integration, we get
 $$P(x_1)=P(x_{\max})+\int_{x_{\max}}^{x_1}P'(u)du,\
 P(x_2)=P(x_{\max})+\int_{x_{\max}}^{x_2}P'(u)du.$$ Making the
 change of variable $v-x_{\max}=x_{\max}-u$, we obtain
 $$\int_{x_{\max}}^{x_1}P'(u)du=-\int_{x_1}^{x_{\max}}|P'(u)|du<-\int_{x_{\max}}^{x_2}|P'(v)|dv\le
 P(x_2)-P(x_{\max}).$$ The first equality holds since $P'$ is positive in $(x_{1},x_{\max})$. Hence $P(x_1)<P(x_2)$. We prove
 $P(y_1)<P(y_2)$ in the same way.

 \end{proof}

In the following we present three situations. We use the last two lemmas. In each situation, there is a unique point $x_\alpha$
such that $$A(x_\alpha)=\alpha,\
h_{\rm top}(E_\Phi(\alpha))=H(x_\alpha).$$ We call $x_\alpha$ the
maximizing point. For every $\alpha\in L_{\Phi}$, there could be
one, two or three points $x$ such that $A(x)=\alpha$. The
maximizing point $x_{\alpha}$ is the one which is the nearest to 1/2. In Figures
\ref{figure4}, \ref{figure5} and \ref{figure6}, those  parts of graph of $A$
corresponding to the maximizing points will be traced by solid lines,
other parts  will be traced by dotted lines.

\medskip

{\em Situation I.} $1/2\le x_{\max} <1<x_{\min}$ (see Figure \ref{figure4}).\\
Let $a=0.4$, $b=1$, and $c=2$. Then $x_{\max}=2/3$ and
$x_{\min}=1.6$. The spectrum is continuous. The following hold:\\
\indent 1. $L_\Phi=[A(0),A(x_{\max})]$.\\
\indent 2. The maximizing points lie in $[0,x_{\max}]$.\\
\indent 3. $A(x)$ is strictly monotonic in $[0,x_{\max}]$.

  \begin{figure}
\begin{minipage}[t]{0.55\linewidth}
\centering
\includegraphics[width=6.8cm]{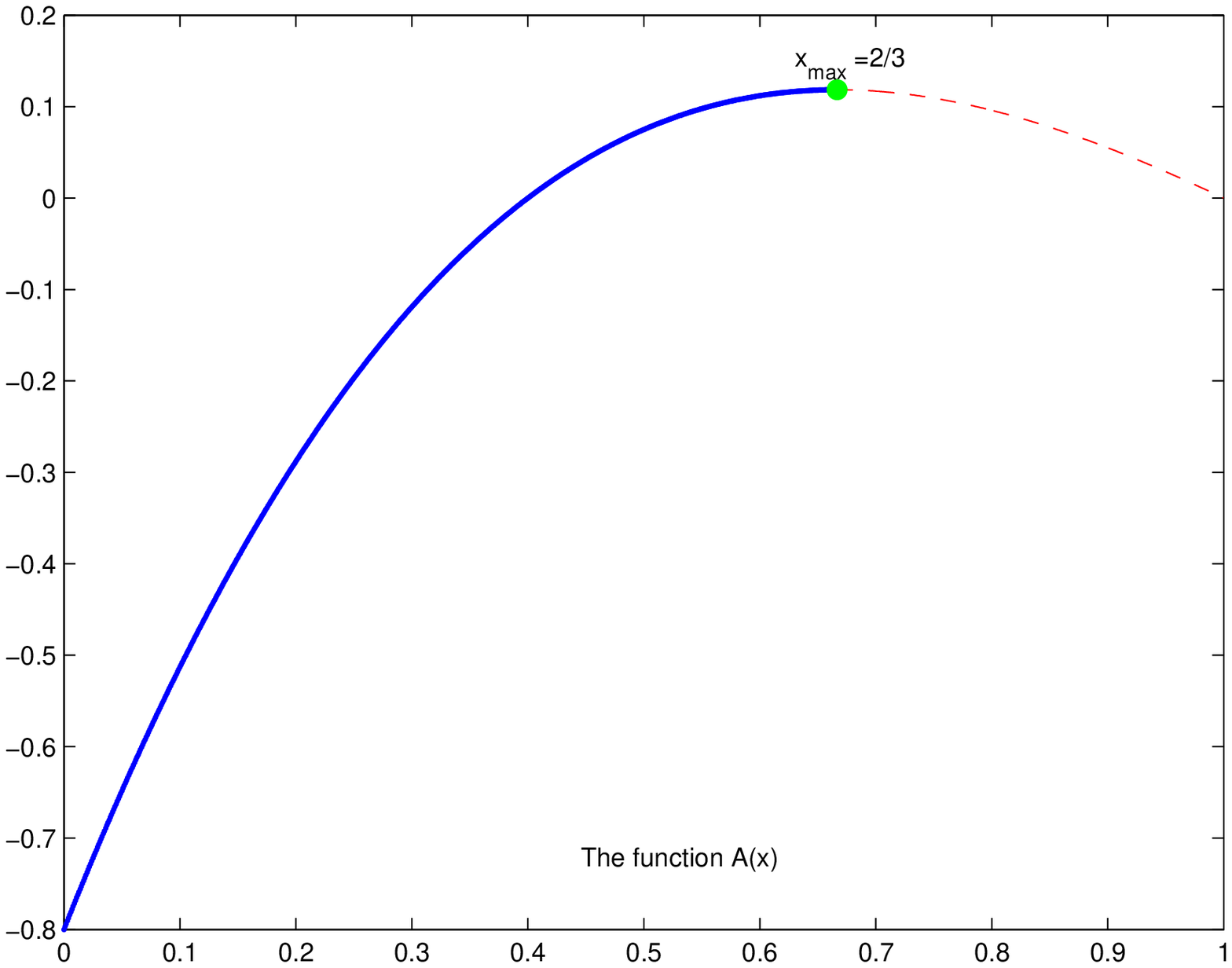}
\label{figure 1.1}
\end{minipage}%
\begin{minipage}[t]{0.55\linewidth}
\centering
\includegraphics[width=6.8cm]{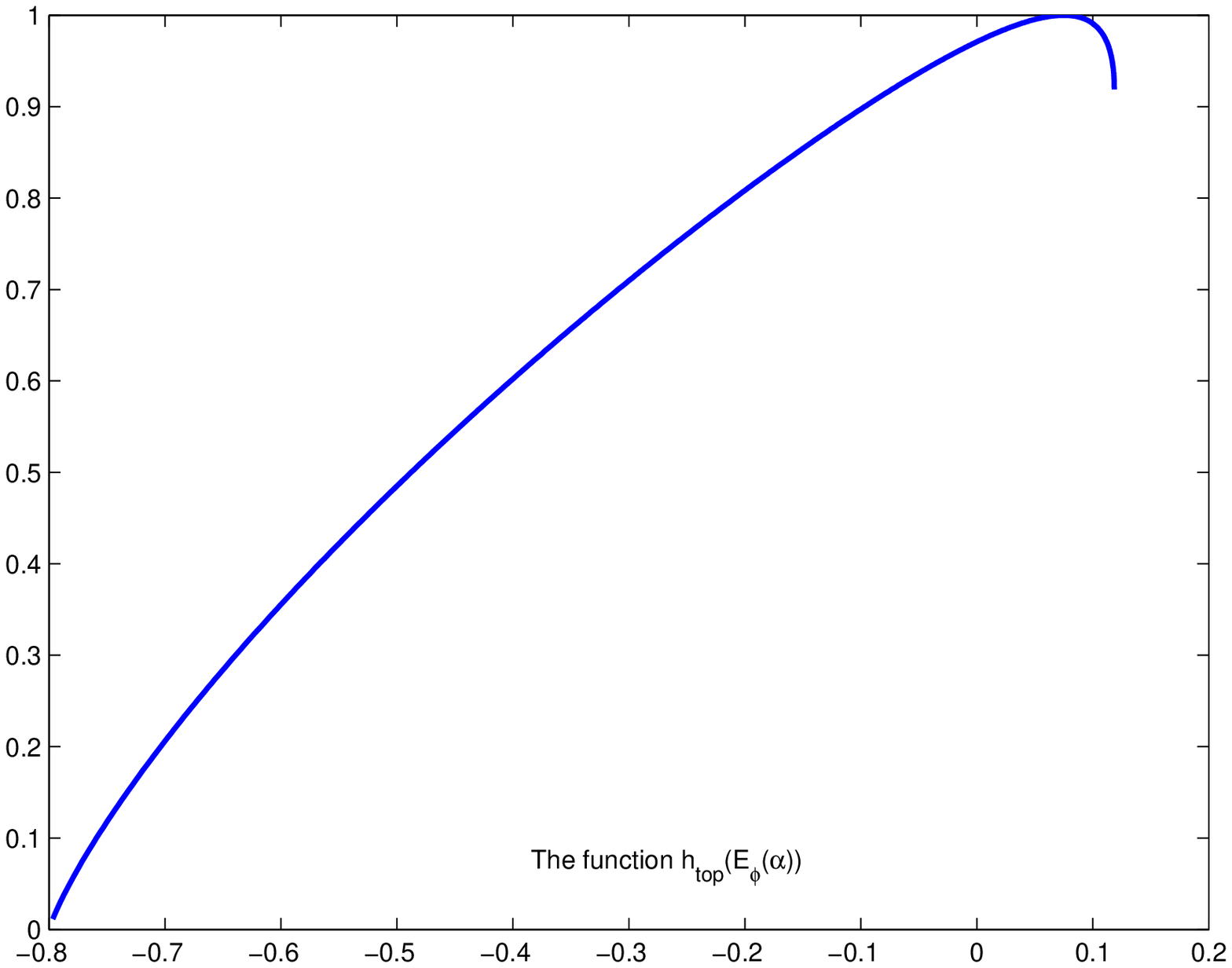}
\end{minipage}
\caption{\Small Situation $1/2< x_{\max} <1<x_{\min}$ ($a=0.4$,
$b=1$, and $c=2$)} \label{figure4}
\end{figure}

\medskip

{\em Situation II.} $1/2\le x_{\max} <x_{\min}<1$ (see Figure \ref{figure5}).\\
Let $a=0.4$, $b=0.7$, and $c=0.8$. Then $x_{\max}=0.5131$,
$x_{\min}=0.7375$ and $x'=0.8737$. The spectrum admits one
discontinuous point. The following hold:\\
\indent 1. $L_\Phi=[A(0),A(1)]$.\\
\indent 2. The maximizing points lie in $[0,x_{\max}]\cup (x',1]$.\\
\indent 3. $A(x)$ is strictly monotonic in each of above two intervals.\\
\indent 4. The spectrum has one discontinuous point at
$A(x_{\max})(=A(x'))$, the entropy jumps from $H(x_{\max})$ to
$H(x')$.

 \begin{figure}
\begin{minipage}[t]{0.55\linewidth}
\centering
\includegraphics[width=6.8cm]{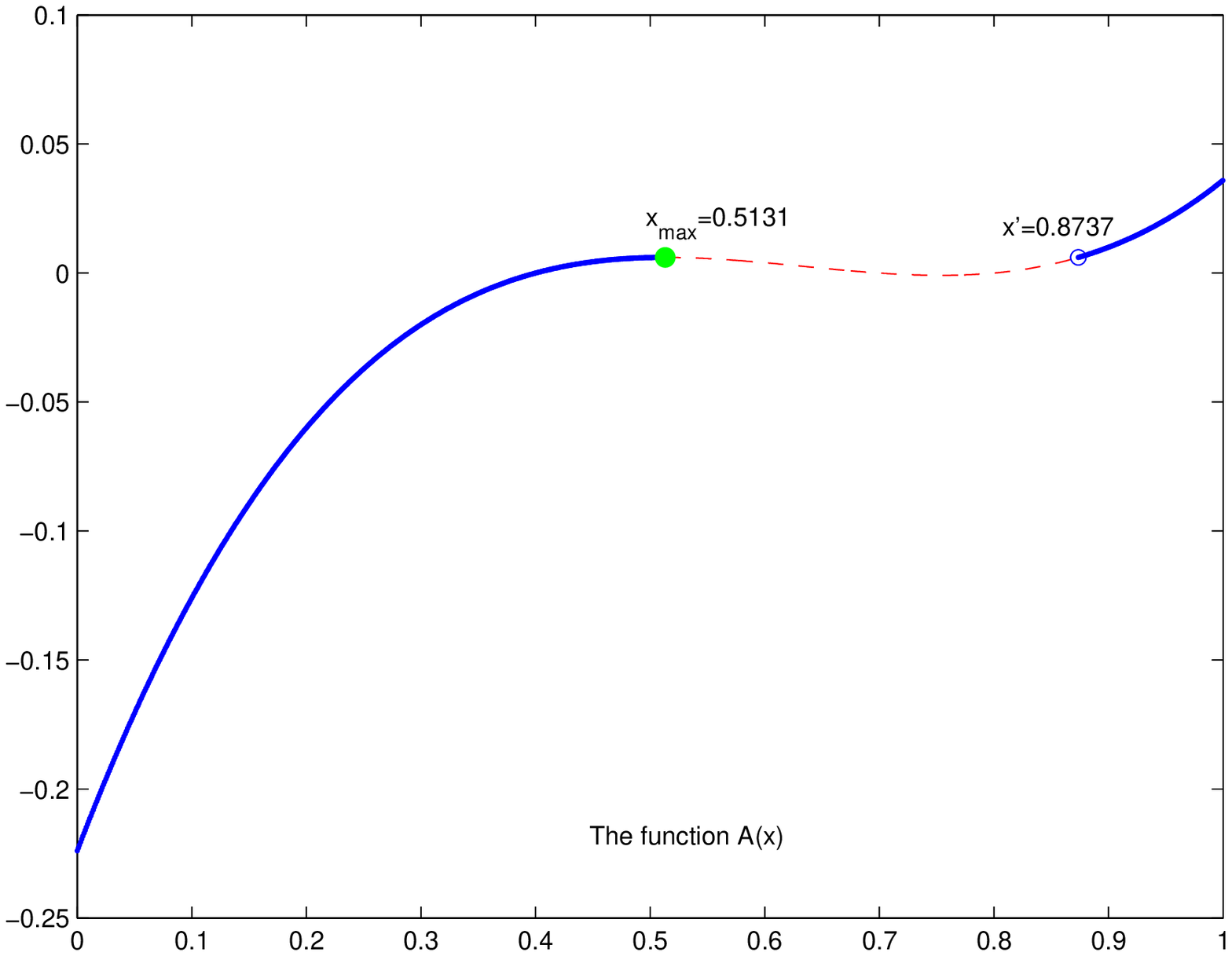}
\end{minipage}%
\begin{minipage}[t]{0.55\linewidth}
\centering
\includegraphics[width=6.8cm]{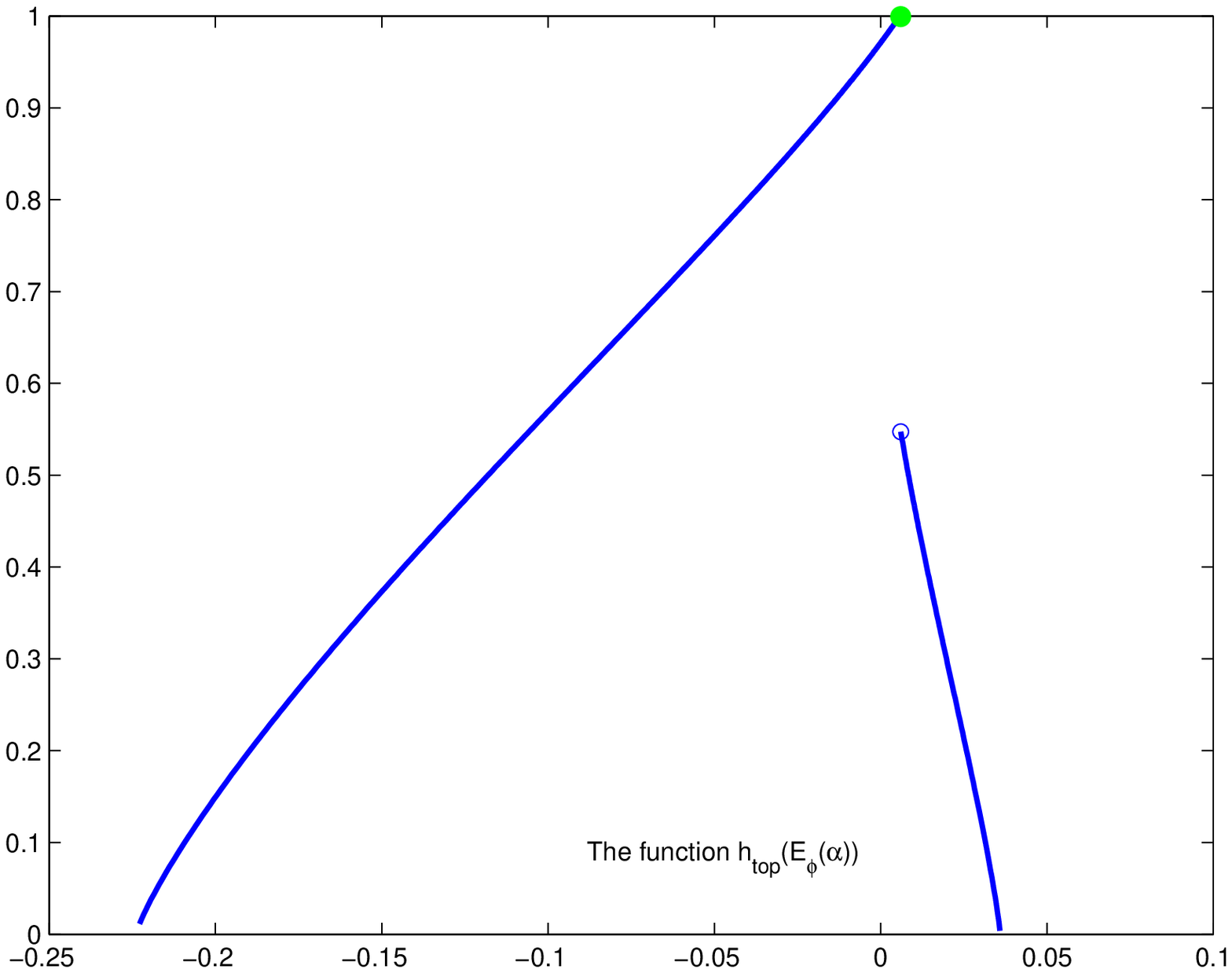}
\end{minipage}
\caption{\Small Situation $1/2< x_{\max} <x_{\min}<1$ ($a=0.4$,
$b=0.7$, and $c=0.8$)} \label{figure5}
\end{figure}

\medskip

{\em Situation III.} $0<x_{\max} <1/2< x_{\min}<1$ (see Figure
\ref{figure6}).\\
 Let $a=0.15$, $b=0.7$, and $c=0.8$. Then
$x_{\max}=0.3479$, $x_{\min}=0.7520$, $x'=0.9541$ and $x''=0.1458$.
The spectrum admits two discontinuous points. The following hold:\\
\indent 1. $L_\Phi=[A(0),A(1)]$.\\
\indent 2. The maximizing points lie in the intervals $[0,x'')\cup [x_{\max},x_{\min}]\cup (x',1]$.\\
\indent 3. $A(x)$ is strictly monotonic in each of above three intervals.\\
\indent 4. The spectrum has two discontinuity points. One is
$A(x'')(=A(x_{\min}))$, where the entropy jumps from $H(x'')$ to
$H(x_{\min})$, the other is $A(x_{\max})(=A(x'))$, where the entropy
jumps from $H(x_{\max})$ to $H(x')$.

 \begin{figure}
\begin{minipage}[t]{0.55\linewidth}
\centering
\includegraphics[width=6.8cm]{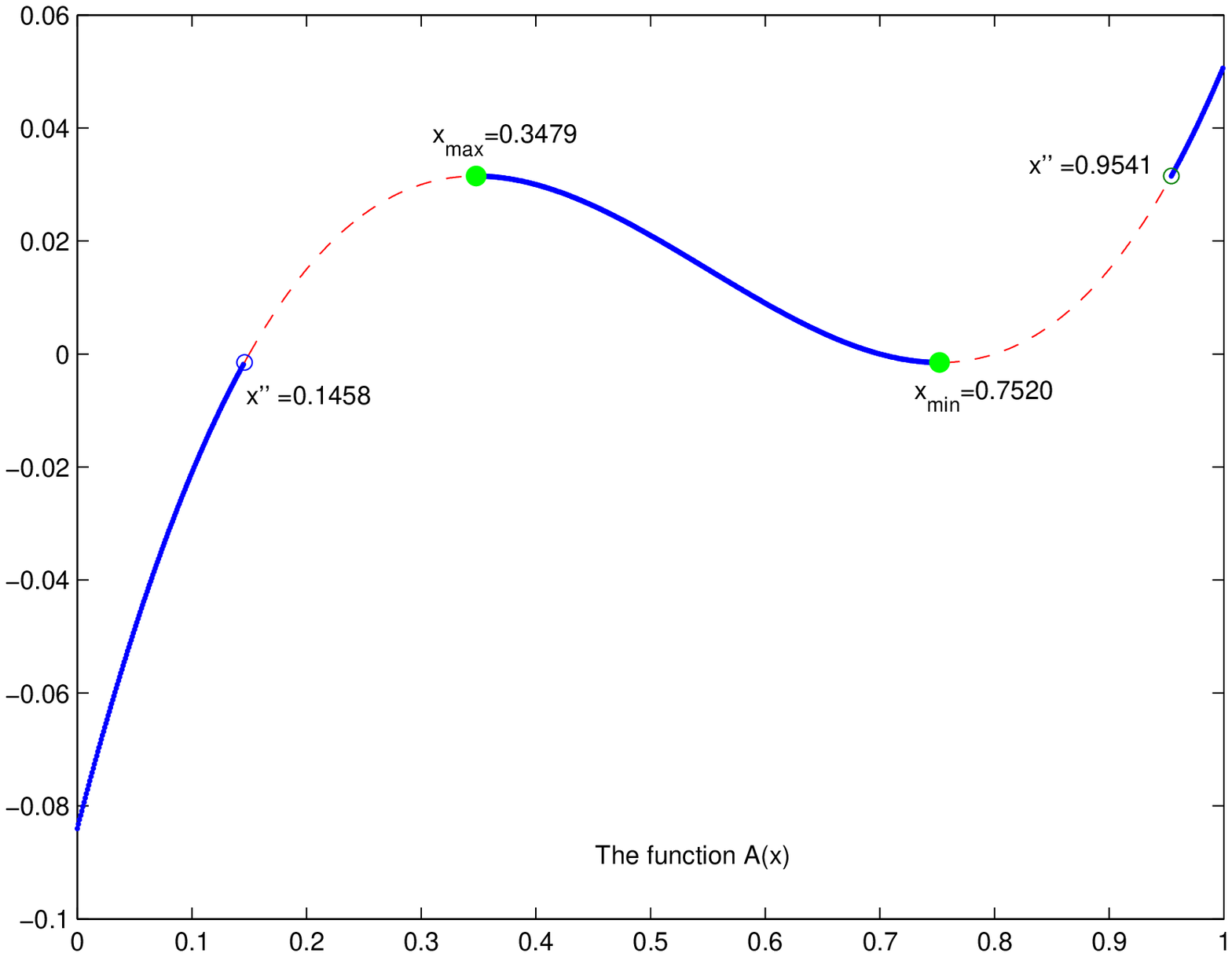}
\end{minipage}%
\begin{minipage}[t]{0.55\linewidth}
\centering
\includegraphics[width=6.8cm]{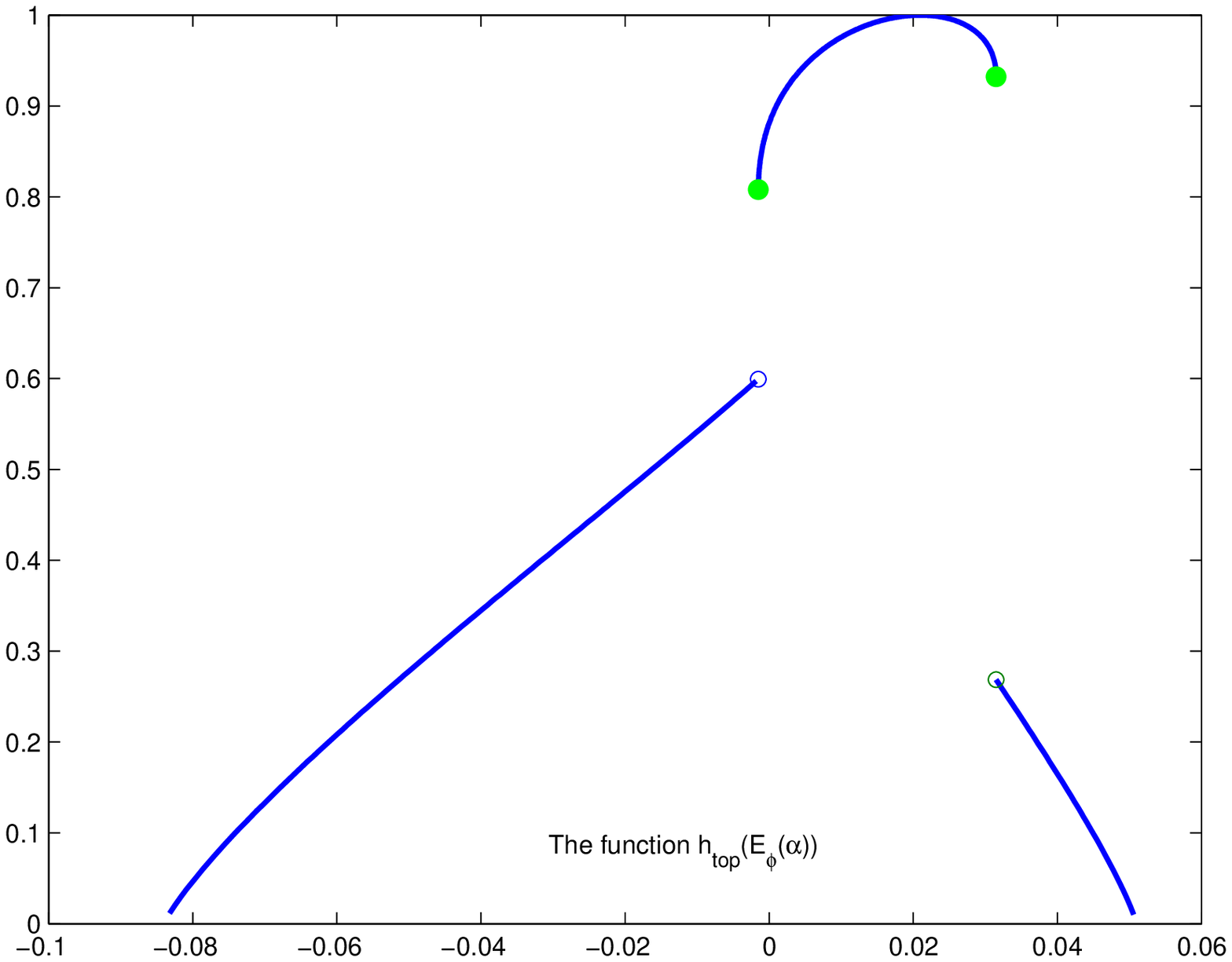}
\end{minipage}
\caption{\Small Situation $0< x_{\max}<1/2 <x_{\min}<1$ ($a=0.15$,
$b=0.7$, and $c=0.8$)} \label{figure6}
\end{figure}


\begin{thebibliography}{99}



\bibitem{ABDGHW} J. Aaronson, R. Burton, H. Dehling, D. Gilat, T. Hill and B.
Weiss,
{\em Strong laws for L- and U-statistics},
\newblock Trans. Amer. Math. Soc., \textbf{348} (1996), 2845--2866.

\bibitem{Assani}
\newblock I. Assani,
\newblock \emph{Multiple recurrence and almost sure convergence for weakly mixing dynamical systems},
\newblock Israel. J. Math,  \textbf{1-3} (1987), 111--124.

\bibitem{Barreira} L. Barreira,
``Dimension and recurrence in hyperbolic dynamics,''
\newblock Progress in Mathematics. Soc., \textbf{272}. Birkh\"{a}user Verlag, Basel, 2008.

 \bibitem{BSS}
 \newblock L. Barreira, B. Saussol, J. Schmeling,
 \newblock \emph{Higher-dimensional multifractal analysis},
 \newblock J. Math. Pures Appl.,  \textbf{81} (2002), 67--91.
 

\bibitem{Bergelson}
\newblock V. Bergelson,
\newblock \emph{Weakly mixing PET},
\newblock Ergod. Th. Dynam. Sys.,  \textbf{3} (1987), 337--349.

\bibitem{Blokh}
\newblock A. M. Blokh,
\newblock \emph{Decomposition of dynamical systems on an interval},
\newblock Usp. Mat. Nauk, \textbf{38} (1983), 179--180.

\bibitem{Bourgain}
\newblock J. Bourgain,
\newblock \emph{Double recurrence and almost sure convergence},
\newblock J. Reine Angew. Math., \textbf{404} (1990), 140--161.

\bibitem{Bowen}
\newblock R. Bowen,
\newblock \emph{Topological entropy for noncompact sets},
\newblock Trans. Amer. Math. Soc., \textbf{184} (1973), 125--136.

\bibitem{DGS}
\newblock M. Denker, C. Grillenberger and K. Sigmund,
\newblock ``Ergodic Theory on Compact Spaces,"
\newblock Springer-Verlag, Berlin-New York, 1976.



\bibitem{Fan1994} A.H. Fan, {\em Sur les dimension de mesures}, Studia Math., \textbf{111} (1994), 1-17.

\bibitem{FFW}
\newblock A.H. Fan, D. J. Feng and J. Wu,
\newblock \emph{Recurrence, entropy and dimension},
\newblock J.  London Math. Soc. \textbf{64} (2001), 229--244.

\bibitem{FLM} A.H. Fan, L. M. Liao and J. H. Ma, {\em Level sets of multiple ergodic averages}.
Monatshefte f\"ur Mathematik, 2011 online.

\bibitem{FLP} A.H. Fan, L. M. Liao and J. Peyri{\`e}re, {\em Generic points in systems of specification and Banach
valued Birkhff ergodic average}, DCDS, \textbf{21} (2008), 1103-1128.

\bibitem{FSW} A.H. Fan, J. Schmeling and M. Wu, {\em Multifractal analysis of multiple ergodic averages},
Comptes Rendus Math\'ematique,
Volume 349, num\'ero 17--18 (2011), 961--964.

\bibitem{Furstenberg} H. Furstenberg, {\em Ergodic behavior of diagonal measures and a theorem of Szemer\'edi on arithmetic
progressions},
J. d'Analyse Math., \textbf{31} (1977), 204--256.

\bibitem{HK} B. Host and B. Kra, {\em Nonconventional ergodic averages and nilmanifolds},
Ann. Math., \textbf{161} (2005), 397--488.

\bibitem{KPS}
\newblock R. Kenyon, Y. Peres and B. Solomyak,
\newblock \emph{ Hausdorff dimension of the multiplicative golden
mean shift},
\newblock Comptes Rendus Mathematique, volume 349, num\'ero 11--12 (2011), 625--628.

\bibitem{Rudin}
\newblock W. Rudin,
\newblock ``Functional Analysis,"
\newblock McGraw-Hill Book Co., New York-D\"{u}sseldorf-Johannesburg, 1973.


\bibitem{Ruelle}
\newblock D. Ruelle,
\newblock ``Thermodynamic formalism. The mathematical structures of classical equilibrium statistical mechanics,"
\newblock Encyclopedia of Mathematics and its Applications, \textbf{5}. Addison-Wesley Publishing Co., 1978.


\bibitem{Schmel97}
 \newblock J. Schmeling,
 \newblock \emph{Symbolic dynamics for $\beta$-shifts and self-normal numbers},
 \newblock Ergod. Th. Dynam. Sys.,  \textbf{17} (1997), 675--694.


\end{thebibliography}
\end{document}